\documentclass[11pt]{amsart}
\usepackage{amssymb}
\usepackage{amsfonts}

\usepackage{bbm}
\usepackage{mathrsfs}
\usepackage{color}

\usepackage{array}

\usepackage{amsfonts,amsmath, amssymb}

\newcommand{\bea}{\begin{eqnarray}}

\newcommand{\eea}{\end{eqnarray}}

\newcommand{\be}{\begin {equation}}

\newcommand{\ee}{\end{equation}}

\newtheorem{theorem}{Theorem}[section]

\newtheorem{corollary}[theorem]{Corollary}

\newtheorem{condition}[theorem]{Condition}

\newtheorem{remark}[theorem]{Remark}

\newtheorem{conjecture}[theorem]{Conjecture}

\begin{document}

\title{$b$-ary expansions of algebraic numbers}

\author {Xianzu Lin }

\date{ }
\maketitle
   {\small \it College of Mathematics and Computer Science, Fujian Normal University, }\\
    \   {\small \it Fuzhou, {\rm 350108}, China;}\\   $\newline$
      \              {\small \it Email: linxianzu@126.com}

$\newline$
\begin{abstract}
In this paper we give a generalization of the main results in
\cite{ab,ab1} about $b$-ary expansions of algebraic numbers. As a
byproduct we get a large class of new transcendence criteria. One
of our corollaries implies that $b$-ary expansions of linearly
independent irrational algebraic numbers are quite independent.
Motivated by this result, we propose a generalized Borel
conjecture.

\end{abstract}


$\newline$



$\newline$



$\newline$

\section{Introduction}
Let $b\geq2$ be a fixed integer. Any real number $\omega$ has a
unique $b$-ary expansion:
$$\omega=[\omega]+\sum_{i\geq1}^{\infty}a_{i}b^{-i}=[\omega]+0.a_1a_2\cdots,$$ where
$a_{i}\in\{0,1,\cdots,b-1\}$ and the set  $\{i\mid a_{i}\neq
b-1\}$ is infinite. $\omega$ is called $normal$ to base $b$ if,
for every $k\geq1$, every block of $k$ digits from
$\{0,1,\cdots,b-1\}$ occurs in the $b$-ary expansion of $\omega$
with frequency $1/b^{k}$.

A classical theorem of Borel \cite{bo} says that almost all real
numbers are normal to base $b$. In \cite{bo1} Borel made the
conjecture that all irrational algebraic numbers are  normal to
base $b$. It seems that this conjecture is far from the reach of
modern mathematics. Let $p(\omega,n)$ be the number of distinct
blocks of length $n$ occurring in the $b$-ary expansion of
$\omega$. It follows from Borel's conjecture that
$p(\omega,n)=b^{n}$ for any irrational algebraic numbers $\omega$.
But even this corollary is too difficulty. A large breakthrough in
this direction is due to Adamczewski and Bugeaud \cite{ab}. Before
introducing their results, we need some preparations.

We say that an infinite word $\textbf{a}= a_1a_2\cdots$ of
elements from $\{0,1,\cdots,b-1\}$ has $long \ repetition$ if the
following condition is satisfied, where the length of a finite
word $A$ is denoted by $|A |$.
 \begin{condition} \label{ccc}
There exist three sequences of finite nonempty words
$\{A_n\}_{n\geq1}$, $\{A'_n\}_{n\geq1}$, $\{B_n\}_{n\geq1}$ such
that:
   \begin{enumerate}
\item for any $n\geq1$, $A_nB_nA'_nB_n$ is a prefix of
$\textbf{a}$;
 \item
the sequence $\{|B_n|\}_{n\geq1}$ is strictly increasing;
 \item
there exists a positive constant $L$ such that $$(|
A_n|+|A'_n|)/|B_n|\leq L,$$ for every $n\geq1$.
\end{enumerate}
\end{condition}

One of the main results in \cite{ab} is the following:
 \begin{theorem} \label{main1}
The $b$-ary expansion of an irrational algebraic number has no
long repetitions.
\end{theorem}
It follows directly from this theorem that
$$\lim_{n\rightarrow +\infty} \frac{p(\omega,n)}{n}=+\infty,$$ where $\omega$ is an irrational algebraic
numbers. This result, though far from the conjecture that
$p(\omega,n)=b^{n}$, is indeed a great advance comparing with the
previous result \cite{fm} that
$$\lim_{n\rightarrow +\infty} p(\omega,n)-n=+\infty.$$

In  \cite{ab1}, Adamczewski and Bugeaud further explored the
independence of $b$-ary expansions of two irrational algebraic
numbers $\alpha$ and $\beta$.

Let $$\textbf{a}= a_1a_2a_3\cdots$$ and $$\textbf{a}'=
a'_1a'_2a'_3\cdots$$ be two infinite words of elements from
$\{0,1,\cdots,b-1\}$. The following is a condition about the pair
$(\textbf{a},\textbf{a}')$:

 \begin{condition} \label{c1}
There exist three sequences of finite nonempty words
$\{A_n\}_{n\geq1}$, $\{A'_n\}_{n\geq1}$, $\{B_n\}_{n\geq1}$ such
that:
   \begin{enumerate}
\item  for any $n\geq1$, the word $A_nB_n$ is a prefix of
$\textbf{a}$ and the word $A'_nB_n$ is a prefix of $\textbf{a}'$;
 \item
the sequence $\{|B_n|\}_{n\geq1}$ is strictly increasing; \item
there exists a positive constant $L$ such that $$(|
A_n|+|A'_n|)/|B_n|\leq L,$$ for every $n\geq1$.
\end{enumerate}
\end{condition}

The main result in \cite{ab1} is:
 \begin{theorem} \label{main2}
Let $\alpha$ and $\alpha'$ be two irrational algebraic numbers. If
their $b$-ary expansions $$\alpha=[\alpha]+0.a_1a_2\cdots,$$ and
$$\alpha'=[\alpha']+0.a'_1a'_2\cdots$$ satisfy Condition \ref{c1},
then the two infinite words $$\textbf{a}= a_1a_2a_3\cdots$$  and
$$\textbf{a}'= a'_1a'_2a'_3\cdots$$ have the same tail.
\end{theorem}

In this paper, we show that for a fix irrational algebraic number
$\alpha$, and a fix nontrivial linear equation
$$a_{1}x_{1}+a_{2}x_{2}+\cdots+a_{n}x_{n}=0,$$ where $a_{1},a_{2},\cdots,a_{n}\in\mathbb{Z}$, the $b$-ary
expansion of $\alpha$ can not have $n$ disjoint long sub-words
which are correlated by the above linear equation. When the linear
equation is $x_{1}-x_{2}=0$, we recover Theorem \ref{main1}.
Similar result holds for several algebraic numbers
$\alpha_{1},\alpha_{2},\cdots,\alpha_{n}$ such that
$1,\alpha_{1},\alpha_{2},\cdots,\alpha_{n}$ are linearly
independent over $\mathbb{Q}$. Applying this  result to a pair of
algebraic numbers quickly implies Theorem \ref{main2}. Our method
is different from that of \cite{ab,ab1}. In particular, we do not
use rational approximations to algebraic numbers. Instead, we
deduce all things from  a single theorem about greatest common
divisor of a big sum and a pow of $b$ (Theorem \ref{main4}).

This paper is structured as follows: In Section 2, we state the
main results of this paper after some preparations. In Section 3,
we supply all the proofs.  In Sections 4, we propose a generalized
Borel conjecture and some other questions.

 \section{Main results}
Throughout this paper, let $b\geq2$ be a fixed integer. All the
irrational algebraic numbers we consider lie in the interval
$(0,1)$. All the words considered in this paper mean words of
elements from $\{0,1,\cdots,b-1\}$. For an irrational algebraic
number $\alpha$, we always identify the $b$-ary expansions
$$\alpha=0.a_1a_2a_3\cdots$$  with the infinite word $$\textbf{a}=
a_1a_2a_3\cdots.$$ The length of a finite word $A$ is denoted by
$| A |$.  For any real number $a$, $[a]$ and $\{a\}$ denote
respectively the integer part and the fractional part of $a$.
Finally, for two integers $a$ and $b$, denote the greatest common
divisor of  $a$ and $b$ by $G.C.D(a,b)$.

Before giving the main result, we introduce some new definitions
and notations.

Given $n$ positive integers $a_1,a_2,\cdots, a_n$, an array of
\textbf{nonzero} rational numbers $\{b_{i,j}\}$, where $1\leq
i\leq n$, $1\leq j\leq a_i$, will be called an $(a_1,a_2,\cdots,
a_n)$-$array$. For each $(a_1,a_2,\cdots, a_n)$-array
$\{b_{i,j}\}$, Set
$$\lfloor
\{b_{i,j}\}\rfloor=\min_{i,j}(b_{i,j}),$$  $$\langle
\{b_{i,j}\}\rangle=\min_{1\leq i\leq n \atop 1\leq j\leq
a_i-1}(b_{i,j+1}-b_{i,j}), $$ and set
$$\langle
\overline{\{b_{i,j}\}}\rangle=\min_{(i,j)\neq(i',j')}(|b_{i,j}-b_{i',j'}|).$$

 Let $L$ be a positive
real number. An $(a_1,a_2,\cdots, a_n)$-array $\{b_{i,j}\}$ is
$L$-$admissible$ if the following conditions are satisfied :
 \begin{enumerate}
 \item  $b_{i,j}$ is a positive integer for  $1\leq i\leq n$,
$1\leq j\leq a_i$;
  \item  $b_{i,j}<b_{i,j+1}$ for $1\leq i\leq n$,
$1\leq j\leq a_i-1$;
 \item
$\max_{i,j}(b_{i,j})\leq L \min_{i,j}(b_{i,j})$.
\end{enumerate}

Now we are in the position to state a theorem from which all the
results of this paper can be derived:
 \begin{theorem} \label{main4}
Let $\alpha_{1},\alpha_{2},\cdots,\alpha_{n}$ be $n$ irrational
algebraic numbers and let $\epsilon$, $L$, and $M$ be fix positive
numbers. Let $P$ be an real valued function defined on the set of
all $L$-admissible $(\underbrace{1,\cdots,1}_n)$-arrays such that
$|P(D)|<M$ for every $D$. For each $L$-admissible
$(\underbrace{1,\cdots,1}_n)$-array $D=\{d_{i,1}\}$, set
$$f(D)=\sum_{i=1}^{n}\alpha_{i}
b^{d_{i,1}}.$$ Then if both $\lfloor D\rfloor$ and $\langle
\overline{D}\rangle$  are sufficiently large, we have
$$G.C.D([f(D)+P(D)],b^{\lfloor
D\rfloor})<b^{\epsilon\lfloor D\rfloor}.$$
\end{theorem}
A direct consequence of Theorem \ref{main4} is the following:
 \begin{theorem} \label{main5}
Let $\alpha_{1},\alpha_{2},\cdots,\alpha_{n}$ be $n$ algebraic
numbers, such that $$1,\alpha_{1},\alpha_{2},\cdots,\alpha_{n}$$
are linearly independent over $\mathbb{Q}$. Let $a_1,a_2,\cdots,
a_n$ be $n$ positive integers, let $C=\{c_{i,j}\}$ be a fixed
$(a_1,a_2,\cdots, a_n)$-$array$, and let $\epsilon$, $L$ and $M$
be fixed positive numbers. Let $P$ be an real valued function
defined on the set of all $L$-admissible $(a_1,a_2,\cdots,
a_n)$-arrays such that $|P(D)|<M$ for every $D$. $\newline$ For
each $L$-admissible $(a_1,a_2,\cdots, a_n)$-array $D=\{d_{i,j}\}$,
set
$$f(D)=\sum_{i=1}^{n}\sum_{j=1}^{a_i}\alpha_ic_{i,j}
b^{d_{i,j}}.$$ Then when both $\lfloor D\rfloor$ and $\langle
D\rangle$  are sufficiently large, we have
$$G.C.D([f(D)+P(D)],b^{\lfloor
D\rfloor})<b^{\epsilon\lfloor D\rfloor}.$$
\end{theorem}

The following two special cases of Theorem \ref{main5} is more
convenient for applications.

 \begin{theorem} \label{main6}
Let $\alpha$ be an irrational algebraic number. Let $n$ a positive
integer. Let $C=\{c_{i}\}$ be a fixed $(n)$-$array$, and let
$\epsilon$ and $L$  be fixed positive numbers. For each
$L$-admissible $(n)$-array $D=\{d_{i}\}$, set
$$f(D)=\sum_{i=1}^{n}c_{i}[\alpha
b^{d_{i}}].$$ Then when both $\lfloor D\rfloor$ and $\langle
D\rangle$  are sufficiently large, we have
$$G.C.D([f(D)],b^{\lfloor D\rfloor})<b^{\epsilon\lfloor D\rfloor}.$$
\end{theorem}

 \begin{theorem} \label{main7}
Let $\alpha_{1},\alpha_{2},\cdots,\alpha_{n}$ be $n$ algebraic
numbers, such that $1,\alpha_{1},\alpha_{2},\cdots,\alpha_{n}$ are
linearly independent over $\mathbb{Q}$. Let  $\epsilon$ and $L$ be
fixed positive numbers. For each $L$-admissible
$(\underbrace{1,\cdots,1}_n)$-array $D=\{d_{i}\}$, set
$$f(D)=\sum_{i=1}^{n}[\alpha_{i}
b^{d_{i}}].$$ Then when $\lfloor D\rfloor$ is sufficiently large,
we have
$$G.C.D(f(D),b^{\lfloor D\rfloor})<b^{\epsilon\lfloor D\rfloor}.$$
\end{theorem}

We need some preliminaries before giving more specific corollaries
of Theorem \ref{main5}.

Let $a_1,a_2,\cdots a_m$ be $m$ fixed non-zero integers. Let
$\textbf{a}$ be an infinite word of elements from
$\{0,1,\cdots,b-1\}$. For any finite non-empty word $s_{m}\cdots
s_{1}s_{0}$, set
$$\overline{s_{m}\cdots s_{1}s_{0}}=\sum_{i=0}^{m}s_{i}b^{i}.$$
The following is a condition about $\textbf{a}$ and the numbers
$a_1,a_2,\cdots a_m$:

 \begin{condition} \label{c3}
There exist $2m$ sequences of finite nonempty words
$$\{A^{1}_i\}_{i\geq1},\cdots,\{A^{m}_i\}_{i\geq1},
\{B^{1}_i\}_{i\geq1},\cdots,\{B^{m}_i\}_{i\geq1}$$ such that:
   \begin{enumerate}
\item  for each $i\geq1,j\leq m$, the word $A^{j}_iB^{j}_i$ is a
prefix of the word $\textbf{a}$ ;
 \item
for each $i\geq1$, $$|B^{1}_i|=\cdots=| B^{m}_i|=k_{i},$$ and the
sequence $\{k_{i}\}_{i\geq1}$ is strictly increasing;
 \item
for each $i\geq1$, $$|A^{1}_i|<\cdots<|A^{m}_i|,$$ and for each
$j< m$ the sequence $\{| A^{j+1}_i|-| A^{j}_i|\}_{i\geq1}$ is
strictly increasing;
 \item
for each $i\geq1$,
$$\sum_{j=1}^{m}a_j\overline{B^{j}_i}\equiv 0\  (mod\ b^{k_{i}});$$
 \item there
exists a positive constant $L$ such that $$| A^{j}_i|/|
B^{j}_i|\leq L$$ for $i\geq1,j\leq m$.
\end{enumerate}
\end{condition}

Let $\textbf{a}^{1},\cdots,\textbf{a}^{m}$ be $m$ infinite words
of elements from $\{0,1,\cdots,b-1\}$. The second condition is
about $\textbf{a}^{1},\cdots,\textbf{a}^{m}$ and the numbers
$a_1,a_2,\cdots a_m$:

 \begin{condition} \label{c5}
There exist $2m$ sequences of finite nonempty words
$$\{A^{1}_i\}_{i\geq1},\cdots,\{A^{m}_i\}_{i\geq1},
\{B^{1}_i\}_{i\geq1},\cdots,\{B^{m}_i\}_{i\geq1}$$ such that:
   \begin{enumerate}
\item  for each $i\geq1,j\leq m$, the word $A^{j}_iB^{j}_i$ is a
prefix of the word $\textbf{a}^{j}$ ;
 \item
for each $i\geq1$, $$|B^{1}_i|=\cdots=| B^{m}_i|=k_{i},$$ and the
sequence $\{k_{i}\}_{i\geq1}$ is strictly increasing;
 \item
for each $i\geq1$,
$$\sum_{j=1}^{m}a_j\overline{B^{j}_i}\equiv 0\  (mod\ b^{k_{i}});$$
 \item there
exists a positive constant $L$ such that $$| A^{j}_i|/|
B^{j}_i|\leq L$$ for $i\geq1,j\leq m$.
\end{enumerate}
\end{condition}

\begin{theorem} \label{T}
Let $\alpha$ be an irrational algebraic number, and let
$a_1,a_2,\cdots a_m$ be $m$ fixed non-zero integers. Then the
$b$-ary expansion of $\alpha$ does not satisfy Condition \ref{c3}.
\end{theorem}

When $m=2$, $a_1=-a_2=1$, the congruence
$$\overline{B^{1}_i}\equiv \overline{B^{2}_i} (mod\ b^{k_{i}})$$ in Condition \ref{c3}
forces $B^{1}_i=B^{2}_i$. In this way we recover Theorem
\ref{main1} immediately.

Theorem \ref{T} and its variants immediately yield a large class
of new transcendence criteria. The following are two simplest
examples:

 Let $ \Phi$ be the set of finite words of length $\geq2$ on the
 alphabet
$\{0,1,\cdots,b-1\}$. Let $s$ be a nonzero integer. For any
$\textbf{a}\in\Phi$ of length $k$, let $\textbf{a}^{1}$ and
$\textbf{a}^{2}$ be the two sub-words of $\textbf{a}$ such that
$\textbf{a}^{1}\textbf{a}^{2}$ is the prefix of $\textbf{a}$ and
$$|\textbf{a}^{1}|=|\textbf{a}^{2}|=[k/2].$$ Let $\textbf{a}^{3}$
be the unique word of length $[k/2]$ in $ \Phi$ satisfying
$$\overline{\textbf{a}^{1}}+\overline{\textbf{a}^{2}}\equiv \overline{\textbf{a}^{3}} (mod\ b^{[k/2]}),$$
and let $\textbf{a}^{4}$ be the unique word of length $k$ in $
\Phi$ satisfying
$$s\overline{\textbf{a}}\equiv \overline{\textbf{a}^{4}} (mod\ b^{k}).$$
We define two operations $f$ and $g$ on $ \Phi$ by
$$f(\textbf{a})=\textbf{a}\textbf{a}^{3}$$ and
$$g(\textbf{a})=\textbf{a}\textbf{a}^{4}.$$ Let $F(\textbf{a})$ be
the limit of $f^{n}(\textbf{a})$, and let $G(\textbf{a})$ be the
limit of $g^{n}(\textbf{a})$ for $n\rightarrow +\infty$. Now
 Theorem \ref{T} immediately implies the following two transcendence
 criteria.
\begin{theorem} \label{TT}
For any $\textbf{a}\in \Phi$, set $$F(\textbf{a})=a_1a_2\cdots$$
and
$$G(\textbf{a})=a'_1a'_2\cdots.$$ Then neither
$$\sum_{i=0}^{\infty}a_{i}b^{-i}=a_{1}b^{-1}+a_{2}b^{-2}+\cdots$$
nor
$$\sum_{i=0}^{\infty}a'_{i}b^{-i}=a'_{1}b^{-1}+a'_{2}b^{-2}+\cdots$$
can be irrational algebraic number.
\end{theorem}

\begin{remark}\label{remar}
In fact, both
$$\sum_{i=0}^{\infty}a_{i}b^{-i}=a_{1}b^{-1}+a_{2}b^{-2}+\cdots$$
and
$$\sum_{i=0}^{\infty}a'_{i}b^{-i}=a'_{1}b^{-1}+a'_{2}b^{-2}+\cdots$$
 in the above theorem are transcendental except for some extreme cases.
\end{remark}

Now we consider simultaneous $b$-ary expansions of several
algebraic numbers.

\begin{theorem} \label{TT1}
Let $\alpha_{1},\alpha_{2},\cdots,\alpha_{m}$ be $m$ algebraic
numbers, such that $$1,\alpha_{1},\alpha_{2},\cdots,\alpha_{m}$$
are linearly independent over $\mathbb{Q}$. Let $a_1,a_2,\cdots
a_m$ be $m$ fixed non-zero integers. Then   Condition \ref{c5} is
not satisfied by the $b$-ary expansions of $\alpha$
$\alpha_{1},\alpha_{2},\cdots,\alpha_{m}$.
\end{theorem}

When $m=2$, $a_1=-a_2=1$,  Condition \ref{c5} reduces to Condition
\ref{c1}.

$\newline$

It can be derived directly from Theorem \ref{main5} that Theorems
\ref{T} and \ref{TT1} have a common generalization of mixed type.
It is a little tedious to write down the full result. A simplest
case says that:

\begin{theorem} \label{TT2}
Let $a_1,a_2,a_3$ be three fixed non-zero integers. Let $\alpha$
and $\alpha'$ be two algebraic numbers, such that
$1,\alpha,\alpha'$ are linearly independent over $\mathbb{Q}$.
Then their $b$-ary expansions $$\alpha=0.a_1a_2\cdots,$$ and
$$\alpha'=0.a'_1a'_2\cdots$$ can not satisfy Condition \ref{c9}
\end{theorem}
 \begin{condition} \label{c9}
There exist six sequences of finite nonempty words
$$\{A^{1}_i\}_{i\geq1},\{A^{2}_i\}_{i\geq1},\{A^{3}_i\}_{i\geq1},
\{B^{1}_i\}_{i\geq1},\{B^{2}_i\}_{i\geq1},\{B^{3}_i\}_{i\geq1}$$
such that:
   \begin{enumerate}
\item  for each $i\geq1$, both $A^{1}_iB^{1}_i$ and
$A^{2}_iB^{2}_i$ are prefixes of $a_1a_2\cdots$ ; \item  for each
$i\geq1$,  $A^{3}_iB^{3}_i$ is a prefix of $a'_1a'_2\cdots$ ;
 \item
for each $i\geq1$, $$|B^{1}_i|=|B^{2}_i|=| B^{2}_i|=k_{i},$$ and
the sequence $\{k_{i}\}_{i\geq1}$ is strictly increasing;
 \item
for each $i\geq1$,
$$\sum_{j=1}^{3}a_j\overline{B^{j}_i}\equiv 0\  (mod\ b^{k_{i}});$$
\item for each $i\geq1$, $|A^{1}_i|<|A^{2}_i|,$ and the sequence
$\{|A^{2}_i|-|A^{1}_i|\}_{n\geq1}$ is strictly increasing;
 \item there
exists a positive constant $L$ such that $$| A^{j}_i|/|
B^{j}_i|\leq L$$ for $i\geq1,j\leq 3$.
\end{enumerate}
\end{condition}

 \section{proofs of the main results}

As in \cite{ab,ab1}, our proofs are dependent upon the following
$p$-adic version of the Schmidt subspace theorem \cite{sc,sch}.
Let $|\cdot|_{p}$ is the $p$-adic absolute value on $\mathbb{Q}$,
normalized by $|p|_{p}=p^{-1}$. we pick an extension of
$|\cdot|_{p}$ to $\overline{\mathbb{Q}}$.

 \begin{theorem} \label{main8}
Let $n\geq1$ be an integer, and let $S$ be a finite set of places
on $\mathbb{Q}$ containing the infinite place. For every
$$\textbf{x}=(x_0,\cdots,x_n)\in\mathbb{Z}^{n+1},$$ set
$$\|\textbf{x}\|=\max_{i}\{|x_{i}|\}.$$ For every $p\in S$,
let $L_{0,p}(\textbf{x}),\cdots,L_{n,p}(\textbf{x})$ be linearly
independent linear forms in $n+1$ variables with algebraic
coefficients. Then for any positive number $\epsilon$, the
solutions $\textbf{x}\in\mathbb{Z}^{n+1}$ of the inequality
$$\prod_{p\in S}\prod_{i=1}^{n+1}|L_{i,p}(\textbf{x})|_{p}\leq\|\textbf{x}\|^{-\epsilon}$$ lie in finitely many proper linear subspaces of
$\mathbb{Q}^{n+1}$
\end{theorem}

\begin{proof}[\textbf{Proof of Theorem \ref{main4}}]
For the infinite place $\infty$ and for each prime $p\mid b$, we
will introduce $n+1$ linearly independent linear forms as follows:

For every
$$\textbf{x}=(x_0,\cdots,x_n)\in\mathbb{Z}^{n+1},$$ set $$L_{0,\infty}(\textbf{x})=-x_0+\sum_{i=1}^{n}\alpha_i
 x_{i},$$ and for each $1\leq
i\leq n$, set $$L_{i,\infty}(\textbf{x})= x_{i}.$$ For each prime
$p\mid b$ and each $0\leq i\leq n$, set $$L_{i,p}(\textbf{x})=
x_{i}.$$

Assume that there exists an infinite sequence
$\{D_{k}\}_{k\geq1}=\{\{d_{i,1,k}\}_{i}\}_{k\geq1}$ of
$L$-admissible $(\underbrace{1,\cdots,1}_n)$-arrays, such that
both
 $\lfloor D_{k}\rfloor$ and $\langle\overline{
D_{k}}\rangle$ tend to infinity, and $$R_{k}\geq
b^{\epsilon\lfloor D_{k}\rfloor},$$  where
$$Q_{k}=[f(D_{k})+P(D_{k})],$$ and $$R_{k}=G.C.D(Q_{k},b^{\lfloor
D_{k}\rfloor}).$$

Set
$$\textbf{X}_{k}=\tfrac{1}{R_{k}}(Q_{k},b^{d_{1,1,k}},\cdots,b^{d_{n,1,k}})\in\mathbb{Z}^{n+1}.$$
 Then we have $$|
 L_{0,\infty}(\textbf{X}_{k})|=|\tfrac{1}{R_{k}}(f(D_{k})-[f(D_{k})+P(D_{k})])|<|\tfrac{2(M+1)}{R_{k}}|.$$
Direct estimation shows that $$\|\textbf{X}_{k}\|<
b^{2\max_{i}(d_{i,1,k})}\leq b^{2L\lfloor
 D_{k}\rfloor},$$ for $k$
sufficiently large.  By product formula (cf.\cite[p.99]{la}),
 $$| L_{i,\infty}(\textbf{X}_{k})|\prod_{p\mid b}| L_{i,p}(\textbf{X}_{k})|_p=| \tfrac{b^{d_{i,1,k}}}{R_{k}}|\prod_{p\mid b}| \tfrac{b^{d_{i,1,k}}}{R_{k}}|_p=1,$$
for $i=1,\cdots,n$.
 Hence $$\prod_{p=\infty \atop or \  p\mid b}
\prod_{0\leq i\leq n}|
 L_{i,p}(\textbf{X}_{k})|_p<|\tfrac{2(M+1)}{R_{k}}|\leq b^{-\tfrac{1}{2}\epsilon\lfloor
D_{k}\rfloor}<\|\textbf{X}_{k}\|^{-\tfrac{\epsilon}{4L}}$$ for $k$
sufficiently large. Now by  Theorem \ref{main8}, there exist a
nonzero element $(e_{0},e_{1},\cdots,e_{n})\in \mathbb{Z}^{n+1}$,
and an infinite subset $\mathbb{N}'$ of $\mathbb{N}$ such that

\begin{equation}\label{for5}e_{0}Q_{k}+\sum_{i=1}^{n}e_{i}
b^{d_{i,1,k}}=0,\end{equation}
 for each $k\in \mathbb{N}'$.

 As  $\langle\overline{
D_{k}}\rangle$ tends to infinity, we can choose an infinite subset
$\mathbb{N}''$ of $\mathbb{N}'$, and a permutation
$\{s_1,\cdots,s_n\}$
 of $\{1,\cdots,n\}$, satisfying:
   \begin{enumerate}
\item  for each $k\in \mathbb{N}''$,
$$d_{s_1,1,k}>\cdots>d_{s_n,1,k};$$
 \item
 for each $i<n$,  $$\lim_{k\in \mathbb{N}'' \atop
k\rightarrow\infty} d_{s_i,1,k}-d_{s_{i+1},1,k}=+\infty.$$
\end{enumerate}
Now dividing Formula (\ref{for5}) by $b^{d_{s_1,1,k}}$ and letting
$k$ tend to infinity along  $\mathbb{N}''$, we obtain
$$\alpha_{s_1}e_{0}+e_{s_1}=0.$$ Hence, by the  irrationality of $\alpha_{1}$,  $$e_{0}=e_{s_1}=0.$$ Finally,
dividing Formula (\ref{for5}) by
$b^{d_{s_2,1,k}},\cdots,b^{d_{s_n,1,k}} $ in turn, we get
$$e_{s_2}=\cdots=e_{s_n}=0.$$ This concludes the proof of the
theorem.
\end{proof}
\begin{remark}\label{remariks}
The above proof uses only the irrationality of $\alpha_{s_1}$.
Thus if we require all the $L$-admissible
$(\underbrace{1,\cdots,1}_n)$-arrays $\{d_{i,1}\}$ to satisfy
$$d_{1,1}>d_{2,1}>\cdots > d_{n,1},$$
 then Theorem \ref{main4} still holds if we only assume the  irrationality of $\alpha_{1}$.
\end{remark}

\begin{proof}[\textbf{Proof of Theorem \ref{main5}}]
Assume that there exists an infinite sequence
$$\{D_{k}\}_{k\geq1}=\{\{d_{i,j,k}\}_{i,j}\}_{k\geq1}$$ of $L$-admissible
$(a_1,a_2,\cdots, a_n)$-arrays, such that both
 $\lfloor D_{k}\rfloor$ and $\langle
D_{k}\rangle$ tend to infinity, and
\begin{equation}\label{for6}R_{k}\geq b^{\epsilon\lfloor
D_{k}\rfloor},\end{equation}   where we set
 $$R_{k}=G.C.D([f(D_{k})+P(D_{k})],b^{\lfloor
D_{k}\rfloor}).$$ We can choose an infinite subset $\mathbb{N}'$
of $\mathbb{N}$,  a partition of the set $$\Phi=\{(i,j)|1\leq
i\leq n,1\leq j\leq a_i\}=\bigcup_{s=1}^{t}U_{s},$$ and a fixed
$(i_{s},j_{s})\in U_{s}$ for each $s$, such that for each $k\in
\mathbb{N}'$,
   \begin{enumerate}
\item  $\min_{(i,j)\in U_{s}}(d_{i,j,k})=d_{i_{s},j_{s},k}$,
$0<s\leq t$;
 \item
$d_{i,j,k}-d_{i_{s},j_{s},k}$ is independent of $k$, when $0<s\leq
t$ and $(i,j)\in U_{s}$;
 \item
 the sequence $\{\mid d_{i,j,k}-d_{i_{s},j_{s},k}\mid\}_{k\in
\mathbb{N}'}$ tends to infinity, when $0<s\leq t$ and $(i,j)\notin
U_{s}$.
\end{enumerate}
As $\langle D_{k}\rangle$ tend to infinity, we see that
$(i,j),(i',j')\in U_{s}$ and $(i,j)\neq(i',j')$ imply $i\neq i'$.
Hence by the assumption on
$\alpha_{1},\alpha_{2},\cdots,\alpha_{n}$,
$$\beta_s=\sum_{(i,j)\in U_{s}}\alpha_ic_{i,j}
b^{d_{i,j,k}-d_{i_{s},j_{s},k}},$$ is an irrational algebraic
number when $k\in \mathbb{N}'$. For $k\in \mathbb{N}'$ we have
$$f(D_{k})=\sum_{s=1}^{t}\beta_s d_{i_{s},j_{s},k}.$$ Set $E_{k}=\{d_{i_{s},j_{s},k}\}_s$. Then $\lfloor E_{k}\rfloor=\lfloor D_{k}\rfloor$  when $k\in \mathbb{N}'$. Hence both
 $\lfloor E_{k}\rfloor$ and $\langle\overline{
E_{k}}\rangle$ tend to infinity along $\mathbb{N}'$. Now applying
Theorems \ref{main4} to $\beta_1,\cdots,\beta_t$ and
$\{E_{k}\}_{k\in \mathbb{N}'}$ implies
$$R_{k}<b^{\epsilon\lfloor
E_{k}\rfloor}= b^{\epsilon\lfloor D_{k}\rfloor},$$
 when $k\in \mathbb{N}'$ is sufficiently large.
This contradicts inequality (\ref{for6}).
\end{proof}

\begin{proof}[\textbf{Proofs of Theorems \ref{main6} and \ref{main7}}]
Theorem \ref{main6} follows by applying Theorem \ref{main5} to the
case $n=1$ and
$$P(D)=P(\{d_{j}\})=-\sum_{j=1}^{a_1}(c_{j}\{\alpha b^{d_{j}}\}+\{c_{j}[\alpha b^{d_{j}}]\}).$$
Theorem \ref{main7} follows by applying Theorems \ref{main5} to
the case $$a_1=\cdots=a_n=1$$ and
$$P(D)=P(\{d_{i}\})=-\sum_{i=1}^{n}\{\alpha_i
b^{d_{i}}\}.$$
\end{proof}

\begin{proof}[\textbf{Proofs of Theorems \ref{T} and \ref{TT1}}]
We prove Theorem \ref{T} first.  Assume that the $b$-ary expansion
of $\alpha$ satisfies Condition \ref{c3}. Set
$d_{i,j}=|A^{j}_iB^{j}_i|$. Then each $D_{i}=\{d_{i,j}\}_j$ is  an
$(L+1)$-admissible $(m)$-array by (\romannumeral5) of Condition
\ref{c3}. It follows from (\romannumeral2) and (\romannumeral3) of
Condition \ref{c3} that both the sequences $\{\lfloor
D_{i}\rfloor\}_{i\geq1}$ and $\{\langle D_{i}\rangle\}_{i\geq1}$
are strictly increasing. (\romannumeral4) of Condition \ref{c3}
implies
$$f(D_{i})=\sum_{j=1}^{m}a_{j}[\alpha  b^{d_{i,j}}]$$ is divisible
by $b^{k_{i}}$, hence
$$G.C.D(f(D_{i}),b^{\lfloor D_{i}\rfloor})\geq b^{k_{i}}\geq b^{\tfrac{1}{L+1}\lfloor
D_{i}\rfloor}$$ for each $i\geq1$. On the other hand, applying
Theorem \ref{main6} to $f(D_{i})$ and the sequence $\{
D_{i}\}_{i\geq1}$ implies $$G.C.D(f(D_{i}),b^{\lfloor
D_{i}\rfloor})< b^{\tfrac{1}{L+1}\lfloor D_{i}\rfloor},$$ when $i$
is sufficiently large. This concludes the proof of Theorem
\ref{T}. Theorem \ref{TT1} follows from Theorem \ref{main7} in
exactly the same way.
\end{proof}

\begin{proof}[\textbf{Proof of Theorem \ref{main2}}]
Assume that the $b$-ary expansions of $\alpha$ and $\alpha'$
satisfy Condition \ref{c1}. It follows from Theorem \ref{TT1} that
$1,$ $\alpha$ and $\alpha'$ are not linearly independent over
$\mathbb{Q}$. Thus there exists a nonzero element $(x,y,z)$ in
$\mathbb{Z}^{3}$, such that
\begin{equation}\label{for7}x+y\alpha+z\alpha'=0.\end{equation}
 The above equality implies $y,z\neq0$.
Set $|A_nB_n|=l_{n},$ $|A'_nB_n|=l'_{n},$ and $|B_n|=k_{n}.$
Without loss of generality, we can choose an infinite subset
$\mathbb{N}'$ of $\mathbb{N}$ such that either $l_{n}-l'_{n}$ is a
fixed  integer $s$, or $\{l_{n}-l'_{n}\}$ tends to $+\infty$ along
$ \mathbb{N}'$. In the second case, applying Theorem \ref{main4}
to the sum $[\alpha b^{l_{n}}]-[\alpha' b^{l'_{n}}]$ implies
$$G.C.D([\alpha b^{l_{n}}]-[\alpha'
b^{l'_{n}}],b^{l'_{n}})<b^{\tfrac{1}{L+1}l'_{n}}\leq b^{k_{n}},$$
when $n\in \mathbb{N}'$ is sufficiently large; this contradicts
(\romannumeral1) of Condition \ref{c1}. Hence we assume that
$l_{n}-l'_{n}$ is a fixed positive integer $s$ for $n\in
\mathbb{N}'$. We have
$$[\alpha b^{l_{n}}]-[\alpha' b^{l'_{n}}]=(\tfrac{z\alpha
b^s+y\alpha+x}{z})b^{l'_{n}}-(\{\alpha b^{l_{n}}\}-\{\alpha'
b^{l'_{n}}\})$$ when $n\in \mathbb{N}'$. If
$$z\alpha b^s+y\alpha\neq0,$$ we can apply Theorem \ref{main4}
to get a contradiction as above. Hence
\begin{equation}\label{for8}z b^s+y=0.\end{equation} Now (\romannumeral1) of Condition \ref{c1} implies
that
$$[\alpha b^{l_{n}}]-[\alpha'
b^{l'_{n}}]=[\tfrac{x}{z}b^{l'_{n}}]-\theta_{n}$$ is divisible by
$b^{k_{n}}$ when $n\in \mathbb{N}'$ is sufficiently large, where
$|\theta_{n}|\leq3$. This implies the $b$-ary expansion of the
rational number $\tfrac{x}{z}$ has arbitrary long block of zero or
$b-1$. On the other hand, it is well-know that the $b$-ary
expansion of a rational number is eventually periodic. This forces
$\tfrac{x}{z}b^i$ to be integer for some positive integer $i$.
This fact combining with equalities $(\ref{for7})$ and
$(\ref{for8})$ implies that the $b$-ary expansions of $\alpha$ and
$\alpha'$ have the same tail.
\end{proof}

 \section{a generalized Borel
conjecture and some other questions} In this section, we collect
some problems and open questions for future work.
 \begin{enumerate}
\item [(i)]  Theorem \ref{TT1} implies that $b$-ary expansions of
linearly independent irrational algebraic numbers are quite
independent. Motivated by this result, we propose a generalized
Borel conjecture.

Let $\xi=(\alpha_{1},\alpha_{2},\cdots,\alpha_{n})$ be $n$-tuple
of real numbers and let {\setlength{\arraycolsep}{0pt}
\begin{eqnarray*}
&&\alpha_{1}=[\alpha_{1}]+0.a_{1,1}a_{1,2}a_{1,3}\cdots\\
&&\alpha_{2}=[\alpha_{2}]+0.a_{2,1}a_{2,2}a_{2,3}\cdots\\
&&\  \  \   \  \  \  \    \  \  \   \  \  \  \   \ \  \  \   \
\vdots
\\
&&\alpha_{n}=[\alpha_{n}]+0.a_{n,1}a_{n,2}a_{n,3}\cdots\\
\end{eqnarray*}
} be their $b$-ary expansions. For two positive integers $m,N$ and
each $n\times m$-matrix $D$ of elements from $\{0,1,\cdots,b-1\}$,
set
$$A_{b}(D,N,\xi):=Card\{i|1\leq i\leq N,D_i=D\},$$
where
\[\begin{matrix}
D_i=\begin{pmatrix}a_{1,i}&a_{1,i+1}&\cdots&a_{1,i+m-1}\\a_{2,i}&a_{2,i+1} & \cdots & a_{2,i+m-1}\\
\hdotsfor{4}
\\a_{n,i}&a_{n,i+1}&\cdots&a_{n,i+m-1}\end{pmatrix}.\\
\end{matrix}\]
The $n$-tuple of real numbers
$\xi=(\alpha_{1},\alpha_{2},\cdots,\alpha_{n})$ is called $normal$
to base $b$ if for each $m\geq1$, each $n\times m$-matrix $D$ of
elements from $\{0,1,\cdots,b-1\}$ occurs in the $b$-ary expansion
of $\xi$ with the frequency $\tfrac{1}{b^{mn}}$, that is, if
$$\lim_{N\rightarrow+\infty}\frac{A_{b}(D,N,\xi)}{N}=\tfrac{1}{b^{mn}},$$ for each $m\geq1$, and each $n\times m$-matrix $D$ of
elements from $\{0,1,\cdots,b-1\}$.

The proof of Borel'theorem by the law of large numbers in
\cite[p.110]{c} can be used directly to show that:
 \begin{theorem} \label{main11}
Almost all $n$-tuples of real numbers are normal to base $b$.
\end{theorem}
Motivated by this result and Theorem \ref{TT1}, we propose the
following generalization of Borel conjecture.
 \begin{conjecture}[Generalized Borel
Conjecture] $\newline$
 Let
$\alpha_{1},\alpha_{2},\cdots,\alpha_{n}$ be $n$ algebraic
numbers, such that $1,\alpha_{1},\alpha_{2},\cdots,\alpha_{n}$ are
linearly independent over $\mathbb{Q}$. Then
$(\alpha_{1},\alpha_{2},\cdots,\alpha_{n})$ is normal to base $b$.
\end{conjecture}

It is easy to check that the above conjecture is invalid if
$1,\alpha_{1},\alpha_{2},\cdots,\alpha_{n}$ are linearly dependent
over $\mathbb{Q}$.
 \item
[(ii)]The proof of Theorem \ref{main4} and Remark \ref{remariks}
implies the following theorem:
 \begin{theorem} \label{main11}
Let $k$ be a positive integer, let $\epsilon$ be a positive
number, and let
$$f(x)=a_kx^k+\cdots+a_1x+a_0$$ be a polynomial with real algebraic coefficients, where $a_k$ is irrational. Then
 we have $$G.C.D([f(b^n)],b^n)<b^{\epsilon n},$$ when
$n\in \mathbb{N}$ is sufficiently large.
\end{theorem}

It is reasonable to expect that the following more general result
also holds:

 \begin{conjecture}
Let $\epsilon$ be a positive number, let $k$, $l$ be two positive
integers,and let
$$f(x)=a_kx^k+\cdots+a_1x+a_0,$$
and
$$g(x)=c_lx^l+\cdots+c_1x+c_0,$$ be polynomials with real algebraic coefficients, where $a_k$
and $c_l$ are linearly independent over $\mathbb{Q}$. Then
 we have $$G.C.D([f(b^n)],[g(b^n)])<b^{\epsilon n},$$ when
$n\in \mathbb{N}$ is sufficiently large.
\end{conjecture}

 \item
[(iii)] All the congruences in Conditions \ref{c3}, \ref{c5} and
\ref{c9} are linear. We ask that wether similar results held for
some nonlinear congruences. According to Theorems
\ref{main4}-\ref{main7}, this amounts to similar estimation of the
upper bound of greatest common divisor of more general power sums.
It is too difficult to give a general formulation of such
question. Following are three simplest illustrations.

 \begin{conjecture}\label{main15}
Let $\alpha$, $\beta$ be two irrational algebraic numbers, and let
$\epsilon$ be a positive number. For nonnegative integers $k$ and
$m$, set
$$f(k,m)=[\alpha b^{k+m}][\beta b^m]+1.$$
 Then
 we have $$G.C.D(f(k,m),b^m)<b^{\epsilon m},$$ when
$m$ is sufficiently large.
\end{conjecture}

 \begin{conjecture}\label{main16}
Let  $\alpha$, $\beta$ and $\epsilon$ be as before. For two
nonnegative integers $k$ and $m$, set
$$f(k,m)=[\alpha b^{k+m}]^2+[\beta b^m].$$
 Then
 we have $$G.C.D(f(k,m),b^m)<b^{\epsilon m},$$ when
$m$ is sufficiently large.
\end{conjecture}

 \begin{conjecture}
Let  $\alpha$, $\beta$ and $\epsilon$ be as before. For positive
integer $m$, set
$$f(m)=[\beta[\alpha b^{m}]^2].$$
 Then
 we have $$G.C.D(f(m),b^m)<b^{\epsilon m},$$ when
$m$ is sufficiently large.
\end{conjecture}

Conjectures \ref{main15} and \ref{main16}, if valid, would provide
new evidences for the Generalized Borel Conjecture.
\end{enumerate}

\end{document}